\documentclass[12pt]{amsart}
\usepackage{amsmath, amssymb, amsthm}
\usepackage[margin=1in]{geometry}

\author{Nikola Kamburov and Boyan Sirakov}

\address{Nikola Kamburov, Facultad de Matem\'aticas, Pontificia Universidad Cat\'olica de Chile, Avenida Vicu\~na Mackenna 4860, Santiago 7820436, Chile}
\email{nikamburov@mat.uc.cl}
\address{Boyan Sirakov, Departamento de Matematica, PUC-Rio, Rua Marques de Sao Vicente 225, G\'avea, Rio de Janeiro -- CEP
22451-900, Brazil}
\email{bsirakov@mat.puc-rio.fr}

\title[Uniform a priori estimates for the Lane--Emden equation in the plane]{Uniform a priori estimates for positive solutions of the Lane--Emden equation in the plane}

\newtheorem{theorem}{Theorem}[section]
\newtheorem{prop}[theorem]{Proposition}

\def\D{\Delta}
\def\O{\Omega}

\def\d{\delta}

\newcommand{\real}{\mathbb{R}}

\newcommand{\del}{\partial}

\begin{document}

\begin{abstract}
We prove that positive solutions of the Lane-Emden equation in a two-dimensional smooth bounded domain are uniformly bounded for all large exponents.
\end{abstract}

\maketitle
\bibliographystyle{plain}

\section{Introduction}
This paper is devoted to the classical Lane--Emden equation
\begin{equation}\label{LE}
\left\{\begin{array}{rcl}
\Delta u  + u^p =0   &\text{in} &\Omega, \quad  p>1  \\
u>0 & \text{in }&\O \\
u  = 0  &\text{on}& \del \Omega.
\end{array}
\right.
\end{equation}
in a bounded domain $\O\subset \real^n$, with a $C^2$-smooth boundary. Our results concern the two-dimensional case, $n=2$.

The Lane--Emden equation models the mechanical structure of self-gravitating spheres, and in particular, the structure of stars in astrophysics -- see \cite{chandrasekhar1957introduction} and \cite{horedt1986exact} for a physical introduction. This equation is also the simplest superlinear equation of second order, and as such has been the object of an enormous number of theoretical studies.

In their classical work \cite{ambrosetti1973dual} Ambrosetti and Rabinowitz showed that a solution of \eqref{LE} can be constructed by viewing the problem as the Euler-Lagrange equation associated with the functional
\begin{equation*}
I[v] = \int_{\O}\frac{|\nabla v|^2}{2} - \frac{|v|^{p+1}}{p+1} ~dx
\end{equation*}
on the Sobolev space $H^1_0(\Omega)$, and by applying the Mountain Pass Theorem.  Alternatively, a solution of \eqref{LE} can be constructed by solving the constrained minimization problem
\begin{equation}\label{CO}
\text{minimize} \quad J[v]:= \frac{1}{2} \int_{\Omega} |\nabla v|^2 ~dx \quad \text{among all} \quad v\in H_0^1(\O) \quad \text{with}\quad \int_{\O} |v|^{p+1} ~dx=1,
\end{equation}
and then scaling the obtained minimizer appropriately. Due to Sobolev embeddings, these approaches work in any space dimension $n$, as long as $p$ is below the critical exponent $p_c$:
\begin{align*}
p_c =  \frac{n+2}{n-2} \quad \text{for} \quad n\geq 3 \quad \text{and} \quad p_c = \infty \quad \text{for} \quad n=2.
\end{align*}
On the other hand, it is classically known \cite{pohozaev} that when $\Omega$ is \emph{star-shaped}, there are no nontrivial solutions in the range $p>p_c$. Thus, it becomes a compelling question to study the behaviour of solutions of \eqref{LE} as $p\nearrow p_c$.

Observe that solutions obtained by \eqref{CO} have \emph{least energy}, among all possible solutions. Whether solutions of \eqref{LE} of higher energies exist depends strongly on the geometry and the topology of the domain: see for instance Dancer \cite{dancer1988effect, dancer1990effect}, Benci--Cerami \cite{benci1991effect}. Their nonuniqueness results gave a negative answer to the original question of Gidas, Ni and Nirenberg (\cite[pp.224]{gidas1979symmetry}) for many types of domains.

On the other hand, it is known, by ODE analysis and the classical symmetry result of \cite{gidas1979symmetry}, that there is a unique solution of \eqref{LE} if $\Omega$ is a ball. A uniqueness result is also available for domains close to a ball (see \cite{zou1994effect}), and for domains $\O\subseteq\real^n$ which are both symmetric and convex in each one of the $n$ orthogonal directions: for all $p>1$ when $n=2$ (\cite{dancer1988effect, DamGroPac}), and  for $p$ close to $p_c$ when $n\ge3$ (\cite{grossi2000uniqueness}).

A long-standing and largely unsolved open question is whether \eqref{LE} has a unique positive solution if $\Omega$ is convex. It was conjectured in \cite{dancer1988effect} that the answer is affirmative. As of today, this conjecture has been verified only if $p$ is close to $1$ or  within the class of least-energy solutions in $2$ dimensions, by C.S. Lin \cite{lin1994uniqueness}.
\bigskip

Returning to the question of asymptotic behaviour as $p$ approaches the critical exponent, it is known that in dimension $n\geq 3$ the solutions of \eqref{LE} blow-up in the $L^{\infty}$-norm as $p\to p_c$, as demonstrated by Rey \cite{rey1990role} and Han \cite{han1991asymptotic}. These papers contain in addition precise asymptotics for the least-energy solutions.

The situation when $n=2$ turns out to be different in this regard. Ren and Wei \cite{ren1994two} showed that \emph{the least-energy solutions stay uniformly bounded as} $p\to \infty$. Whether this property is valid for arbitrary solutions has been an open problem until now.

Our main result gives an affirmative answer.

\begin{theorem}\label{main} Let $\Omega\subset\real^2$ be a bounded domain with $C^2$-boundary and let $p_0>1$. There exists a constant $C$, depending only on $p_0$ and $\Omega$, such that for all $p\geq p_0$ any solution $u_p\in C^{2}(\Omega)\cap C^0(\overline{\O})$ of \eqref{LE} satisfies:
\[
\|u_p\|_{L^{\infty}(\Omega)} \leq C.
\]
\end{theorem}

Observe that we make no assumption on $\Omega$ other than smoothness. That Theorem \ref{main} should hold for convex domains has been a common belief among experts. It may seem surprising that it is actually true for any smooth domain.

The proof of Theorem \ref{main} employs the Green's function representation formula, the Harnack inequality, and a rescaling argument. Additionally, we make use of the approach by de Figueiredo, Lions, and Nussbaum \cite{de1982priori}  to glean information about the near-boundary behaviour of solutions.
\bigskip

Next, we describe some applications of Theorem \ref{main}. In the original result of Ren and Wei the uniform boundedness of least-energy solutions follows (see the proof of Theorem 1.1 in  \cite{ren1994two}) as a consequence of the key integral bound
\begin{equation}\label{intgCond}
p \int_{\O} |\nabla u_p|^2 ~dx \leq C \qquad \text{for all large}\quad p,
\end{equation}
which in \cite{ren1994two} is shown to hold for least-energy solutions.

As a corollary to Theorem \ref{main}, we show that  \eqref{intgCond} holds for arbitrary solutions, provided $\Omega$ is  \emph{strictly star-shaped}.

\begin{theorem}\label{coro_IntgBnd} Let $\O\subset \real^2$ be a strictly star-shaped bounded domain with $C^2$-boundary and let $p_0>1$. There exists a constant $C$, depending only on $p_0$ and $\Omega$, such that for all $p\geq p_0$ any solution $u_p\in C^{2}(\Omega)\cap C^0(\overline{\O})$ of \eqref{LE} satisfies
\[
p \int_{\O} |\nabla u_p|^2 ~dx = p \int_{\O} u_p^{p+1} ~dx \leq C.
\]
\end{theorem}
The proof of this theorem uses the Pohozaev identity, which is the reason behind the requirement on the domain to be star-shaped.
\bigskip

In the recent work \cite{de2017asymptotic} De Marchis, Ianni and Pacella have performed a fairly complete asymptotic description of positive solutions of the Lane--Emden problem \eqref{LE} as $p\to \infty$, under the hypothesis that the integral quantity in \eqref{intgCond} has a limit. The paper \cite{de2017asymptotic} builds on the results of Ren and Wei \cite{ren1994two}, Adimurthi and Grossi \cite{adigrossi}, as well as the earlier work of the same authors \cite{dMIP} where they also obtain asymptotics for sign-changing solutions of the Lane-Emden equation (we refer to \cite{de2017asymptotic} for a full list of references and the history of this problem).

It is shown in \cite{de2017asymptotic} that if $\beta=\lim_{p\to\infty} p \int_{\O} |\nabla u_p|^2$ exists, then  $(u_p)$ concentrate at a finite number of points $\{x_j\}_{j=1}^k \subseteq \O$ as $p\to \infty$, whose locations satisfy an ``electrostatic force'' balance equation. Moreover, for a subsequence $p_i\to\infty$  one has (for a small enough $\d>0$)
\[
\lim_{p_i\to \infty} \|u_{p_i}\|_{L^{\infty}(B_{\d}(x_j))} = m_j, \quad j=1,2,\ldots, k \quad \text{with}\quad m_j\geq \sqrt{e} \quad \text{and}\quad  \beta = 8\pi \sum_{j=1}^k m_j^2
\]
and suitable rescalings of the solutions converge to a positive solution of the classical Liouville equation $-\Delta u = e^u$.
In \cite{de2017asymptotic} the authors conjecture that all $m_j=\sqrt{e}$, so that the asymptotic ``energy'' $\beta = (8\pi e)k$ is quantized.

In a recent preprint \cite{de2018asymptotic} by the three authors together with  Massimo Grossi, they present a proof of this conjecture (and thus a refined asymptotic description) under the sole assumption \eqref{intgCond}. A proof of the quantization conjecture, again under \eqref{intgCond}, was also independently announced by Thizy \cite{thizy2018sharp}.\medskip

Our Theorem \ref{coro_IntgBnd} shows that the assumption \eqref{intgCond} is actually always satisfied in a star-shaped smooth domain. Thus the results in  \cite{de2017asymptotic, de2018asymptotic}, combined with Theorem \ref{coro_IntgBnd},
provide a full asymptotic description for solutions of \eqref{LE} in such domains.\medskip

Furthermore, from the presentations \cite{Pac} and \cite{FdM} we learned of a joint work in progress of the authors of \cite{de2018asymptotic}, in which they claim to have demonstrated that in convex domains, solutions of \eqref{LE} satisfying condition \eqref{intgCond} have Morse index equal to $1$ (as critical points of the functional $I[v]$), if $p$ is sufficiently large depending on the constant $C$ in \eqref{intgCond}. On the other hand, the proof of the main result of C.S.\! Lin \cite{lin1994uniqueness} implies that the Lane--Emden problem, when set in a bounded convex domain in the plane, has a unique positive solution of \mbox{Morse index $1$.} Thus, their work in combination with our Theorem \ref{coro_IntgBnd} would establish the uniqueness of positive solutions of the Lane--Emden problem in smooth  convex domains in the plane for all sufficiently large $p$.




\section{Proofs}

Everywhere in the sequel the letters $C,c$ (possibly with indices) will denote positive constants which depend only on $p_0$ and $\Omega$, and which may change from line to line.

We start with the following auxiliary proposition.

\begin{prop}\label{aux} Under the assumptions of  Theorem \ref{main}, there exist positive constants $\d$ depending only on  $\Omega$, and $C$ depending only on $p_0$ and $\Omega$, such that
\begin{itemize}
\item The maximum of $u=u_p$ in $\Omega$ is attained in $\{x\in \Omega:\text{dist}(x,\del\Omega)\geq \d\}$;
\item We have the bound $$\int_{\O}u^p ~dx \leq C.$$
\end{itemize}
\end{prop}
\begin{proof}
By using the moving planes technique in combination with the Kelvin transform as in \cite[pp.223]{gidas1979symmetry}  or \cite[pp.45-52]{de1982priori}, one can show the existence of positive constants $\d$ and $\gamma$ (depending only on $\O$) such that
$$
 \text{ for each } x\in \{x\in \Omega, \text{dist}(x,\del\Omega)<2\d\} \quad  \text{ there exists a measurable set } I_x \text{ such that}
$$
\begin{equation}\label{MovPlane}
\begin{array}{rl}
(i) & I_x \subseteq \{x\in \Omega, \text{dist}(x,\del\Omega)\geq \d\}\\
(ii) & |I_x|\geq \gamma \\
(iii) & u(x)\leq u(\xi) \text{ for all } \xi\in I_x.
\end{array}•
\end{equation}
For details, please see the discussion after Theorem $1.2$ in \cite{de1982priori}. The statements in \eqref{MovPlane} correspond to $(8)^\prime$ in \cite{de1982priori}. Observe that the constant $C$ which appears in $(8)^\prime$(iii) in that paper can be taken to be one, since $n=2$. In particular, this implies that the maximum of $u$ is achieved sufficiently far away from the boundary -- in $\{\text{dist}(x,\del\Omega)\geq \d\}$.

For the second part of the proposition, we argue as in Steps 1 and 2 of \cite[Theorem 1.1]{de1982priori}, tracking the dependence in $p$. Let $\phi$ be the   eigenfunction of the Dirichlet Laplacian, associated with the lowest eigenvalue $\lambda=\lambda(\O)$:
\begin{equation}
-\D \phi = \lambda \phi \qquad \phi = 0 \text{ on } \del\Omega, \quad \text{normalized so that} \quad \|\phi\|_{L^1(\Omega)}=1.
\end{equation}
Multiplying \eqref{LE} by $\phi$ and integrating by parts we get
\begin{equation}\label{eigenMult}
\lambda \int_{\O} u\phi ~dx= \int_{\O} u^p \phi ~dx.
\end{equation}
Let $t_p = (2\lambda)^{1/(p-1)}$, so that we have $u^p > 2\lambda u$ whenever $u>t_p$. Now using \eqref{eigenMult}, we obtain
\begin{align*}
2\lambda \int_{\O} u\phi ~dx &= \int_{u>t_p} (2\lambda u)\phi ~dx +  2\lambda \int_{u\leq t_p}  u\phi ~dx \\ &\leq \int_{\Omega} u^p \phi ~dx + 2\lambda t_p \int_{\O}\phi ~dx
 = \lambda  \int_{\O} u\phi ~dx + 2\lambda t_p,
\end{align*}
which yields
\begin{equation*}
\int_{\Omega} u^p \phi ~dx = \lambda \int_{\O} u\phi ~dx  \leq 2\lambda t_p = (2\lambda)^{1+1/(p-1)}\leq \max\{1,  (2\lambda)^{1+1/(p_0-1)}\}=C(p_0,\Omega).
\end{equation*}
Furthermore, by \eqref{MovPlane} we see that for all $y\in \{\text{dist}(y,\del\Omega)<2\d\}$
\begin{equation*}
C\geq \int_{\O} u^p \phi ~dx \geq \int_{I_y} u^p \phi ~dx \geq \gamma u^p(y) \min_{I_y}\phi.
\end{equation*}
But by the  Harnack inequality $ \min_{I_y}\phi \geq c_1$ with $c_1$ depending only $\Omega$, as $\phi > 0$ in $\Omega$ and $I_y \subseteq \{\text{dist}(y,\del\Omega)>\d\}$. Therefore,
\begin{equation}
u^p(y)\leq C_0 \quad \text{for all} \quad y\in \{\text{dist}(y,\del\Omega)<2\d\}.
\end{equation}
A standard barrier argument applied to
$$0\le -\Delta u\le C_0\quad\mbox{ in } \{\text{dist}(y,\del\Omega)<2\d\}, \qquad u=0 \mbox{ on } \partial \Omega,$$ now gives us
\begin{equation}\label{bdryGrad}
|\nabla u|=|u_{\nu}|\leq C_1 \quad \text{on }\del\Omega,
\end{equation}
where $u_{\nu}$ denotes the outer normal derivative of $u$ at $\del \O$. From \eqref{bdryGrad} we infer  the desired estimate by using the Divergence Theorem:
\[
\int_{\O} u^p ~dx= \int_{\O} -\D u ~dx = \int_{\del \O} u_{\nu} ~ds \leq C_1 |\partial \Omega|=C.
\]
\end{proof}

We  now move to the proof of  Theorem \ref{main}.
\begin{proof}[Proof of Theorem \ref{main}]\

 For the sake of notational simplicity we will drop the subscript $p$ from $u_p$.


Let $M$ denote the maximum of $u$, $M = \max_{\overline{\O}} u$. From Proposition \ref{aux} we know that $M$ is attained at a point which is at least at  distance $\d$ away from the boundary $\del \O$. Without loss of generality, we pick its location to be the origin: $u(0)=M$.  Denote by $B_r\subseteq\real^2$ the open disk of radius $r$, centered at the origin. We may also assume that $\Omega \subseteq B_1$. Indeed if, in fact, $B_R$ were the smallest disk such that $B_R \supseteq \O$ with $R=R(\O)>1$, then we can use the scale-invariance of the equation: the function $\tilde{u}$ defined by
\[
\tilde{u}(x):= R^{2/(p-1)}u(Rx) \quad \text{for } x\in \tilde{\Omega}:= R^{-1}\O \subseteq B_1
\]
is also a non-negative solution of \eqref{LE}. If we know that $\|\tilde{u}\|_{L^{\infty}(\tilde{\O})} \leq C(p_0, \tilde{\O})= C(p_0, \O)$, then
\[
\|u\|_{L^{\infty}(\O)} = R^{-2/(p-1)}\|\tilde{u}\|_{L^{\infty}(\tilde{\O})} \leq C(p_0, \O),
\]
as well.

Let $G(y)$ denote the Green's function for the Laplacian $(-\D)$ with pole at the origin. Then we can write
$$
G(y) = -\frac{1}{2\pi}\log(|y|) - g(y)
$$
where $g$ is harmonic in $\O$ with boundary data $$g(y) = -\frac{1}{2\pi}\log(|y|).$$ Since $1\geq |y|\geq \d$ for all $y\in \del\O$, we have $\|g\|_{L^{\infty}(\del\O)} \leq C$, so that the Maximum Principle yields
\begin{equation}\label{gBnd}
0\leq g \leq C \quad \text{in}\quad  \O.
\end{equation}


Now, applying Green's representation formula (Theorem 2.2.12 in \cite{evansPDE}), we obtain
\begin{align}
M & =u(0) = \int_{\Omega} G(y) u^p(y) ~ dy = \frac{1}{2\pi} \int_{\Omega} \log(1/|y|)\, u^p ~ dy - \int_{\Omega} g u^p dy
 \notag \\ & \geq  \frac{1}{2\pi} \int_{\Omega} \log(1/|y|)\, u^p ~dy - C\int_{\Omega} u^p ~dy
\notag  \\ & \geq   \frac{1}{2\pi} \int_{\Omega} \log(1/|y|)\, u^p ~dy - C',\label{mainEst}
\end{align}
where in the last two lines we used \eqref{gBnd} and the bound on $\int u^p$ of Proposition \ref{aux}.

We want to exploit the logarithmic singularity. Instead of $u$, we shall work with
$$
v(x):=1-\frac{u(M^{-\frac{p-1}{2}}x)}{M} $$
which satisfies
\begin{equation}\label{eqn_v}
\begin{array}{l}
\Delta v = (1-v)^p \quad \text{in} \quad \Omega':= M^{\frac{p-1}{2}} \Omega  \subseteq B_{M^{(p-1)/2}}\\
 v\geq 0 \quad \text{with} \quad v(0) = 0.
\end{array}
\end{equation}
After the change of variables $y=x/M^{\frac{p-1}{2}}$, the estimate \eqref{mainEst} above reads
$$
M+ C' \geq \frac{M}{2\pi} \int_{\Omega'} \log(M^{\frac{p-1}{2}}/|x|) (1-v)^p ~dx
$$
or
$$
 \int_{\Omega'} \log(M^{\frac{p-1}{2}}/|x|) (1-v)^p ~dx \leq 2\pi + \frac{C^\prime}{M}.
$$
If $M<1$ we are done, so we may assume that $M\geq 1$ which guarantees $B_{\d}\subseteq \O'$. By the inhomogeneous Harnack inequality (see for instance Theorem 4.17 in \cite{HanLin}) applied to \eqref{eqn_v} and $0\le (1-v)^p\le1$, $v(0)=0$ we get
 $$\|v\|_{L^{\infty}(B_r)} \leq C r^2,
$$
for all $r\in (0,\delta)$, implying
$$
 v \leq \frac{C_1}{p} \quad \text{in}\quad B_{\d/\sqrt{p}},
$$
 so that
$$(1-v)^p\geq e^{-C_1}=:c_1\quad\mbox{ in }\quad B_{\d/\sqrt{p}}.
$$
 Since $\O' \subseteq B_{M^{(p-1)/2}}$, we have  $\log(M^{\frac{p-1}{2}}/|x|) \geq 0$ in $\Omega'$, thus
\begin{equation}\label{penutm}
2\pi+C^\prime\geq 2\pi + \frac{C^\prime}{M} \geq  \int_{B_{\d/\sqrt{p}}} \log(M^{\frac{p-1}{2}}/|x|) (1-v)^p ~dx \geq c_1 \int_{0}^{\d/\sqrt{p}} \log(M^{\frac{p-1}{2}}/r)\, r\, dr
\end{equation}
It is easy to check that for any $\alpha,\rho>0$
 $$\int_{0}^{\rho} \log(\alpha/r) r ~dr \geq  \frac{1}{2} \log(\alpha/\rho)\rho^2,$$
from which  we see that the right-hand side of \eqref{penutm}
\[
c_1 \int_{0}^{\d/\sqrt{p}} \log(M^{\frac{p-1}{2}}/r) r ~dr \geq \frac{c_2}{p} \log(\sqrt{p}M^{\frac{p-1}{2}}/\d) \geq c_3\left( \frac{\log p}{p} + \frac{p-1}{p} \log M  \right),
\]
implying the desired
\[
\log M \leq C.
\]
\end{proof}

Finally, we prove Theorem  \ref{coro_IntgBnd}.

\begin{proof}[Proof of Theorem \ref{coro_IntgBnd}]\
Without loss of generality, we may assume that $\O$ is star-shaped with respect to the origin:
\begin{equation}\label{star}
(x,\nu) \geq \alpha
\end{equation}
for some constant $\alpha=\alpha(\O)>0$, where $\nu$ denotes the unit outer normal to $\del \O$.
Our argument will be based on the well-known Pohozaev identity (see \cite{pohozaev} or \cite[pp.516]{evansPDE}): any solution of \eqref{LE} is such that
\begin{equation}\label{pohoz}
 \frac{4}{p+1}\int_{\O} u^{p+1} ~dx = \int_{\del \O} (x,\nu) u_{\nu}^2 ~ ds.
\end{equation}
We derive
\begin{align*}
\int_{\O} u^p ~dx = \int_{\del \O} -u_{\nu} ~ds & \leq |\del \O|^{1/2}\left(\int_{\del\O}u_{\nu}^2 ~ ds\right)^{1/2} \\
& \leq [\alpha^{-1}|\del \O|]^{1/2}\left(\int_{\del\O}(x,\nu) u_{\nu}^2 ~ ds\right)^{1/2} \\
& \leq  [\alpha^{-1}|\del \O|]^{1/2} \frac{2}{\sqrt{p+1}} \left(\int_{\O} u^{p+1} ~dx\right)^{1/2} \\
& \leq 2[\alpha^{-1} |\del \O|\|u\|_{L^{\infty}(\O)}]^{1/2} p^{-1/2} \left(\int_{\O} u^{p} ~dx\right)^{1/2}
\end{align*}
where the first inequality follows from Cauchy-Schwarz, the second is due to \eqref{star} and the third  follows fom \eqref{pohoz}. Thus,
\[
p \int_{\O} u^{p} ~dx \leq 4\alpha^{-1} |\del \O|\|u\|_{L^{\infty}(\O)}
\]
so that
\[
p \int_{\O} u^{p+1} ~dx \leq p \|u\|_{L^{\infty}(\O)}  \int_{\O} u^{p} ~dx \leq 4\alpha^{-1} |\del \O|\|u\|_{L^{\infty}(\O)}^2.
\]
Integration by parts and Theorem \ref{main} give us ultimately
\[
 p \int_{\O} |\nabla u|^2 ~dx = p \int_{\O} u(-\D u) ~dx =  p \int_{\O} u^{p+1} ~dx \leq 4 \alpha^{-1} |\del \O|\|u\|_{L^{\infty}(\O)}^2 \leq C,
\]
where $C$ depends only on $p_0$ and $\O$, by Theorem \ref{main}.


\end{proof}\medskip

\paragraph{\emph{Acknowledgements.} } N.K. gratefully acknowledges the hospitality of the Department of Mathematics, PUC-Rio, where part of the work was conducted. N.K. was also partially supported by Proyecto FONDECYT Iniciaci\'on No. 11160981. B.S. was partially supported by CNPq grant 401207/2014-5.


\bibliography{LE2D_Biblio}
\end{document}